\documentclass[a4paper]{article}

\usepackage{fullpage}

\usepackage[english]{babel}
\usepackage[utf8x]{inputenc}
\usepackage[T1]{fontenc}

\usepackage{amsmath,amsthm}
\usepackage{amssymb}
\usepackage{bbm, dsfont}
\usepackage{wasysym}%
\usepackage{color}%
\usepackage{capt-of}

\usepackage[a4paper,top=3cm,bottom=2cm,left=3cm,right=3cm,marginparwidth=1.75cm]{geometry}

\usepackage{amsmath}
\usepackage{graphicx}
\usepackage[colorinlistoftodos]{todonotes}
\usepackage[colorlinks=true, allcolors=blue]{hyperref}

\newtheorem{proposition}{Proposition}
\newtheorem{assumption}{Assumption} 

\newtheorem{corollary}{Corollary}
\newtheorem{remark}{Remark}
\newtheorem{definition}{Definition}

\newcommand{\R}{{\mathbb R}} 
\newcommand{\N}{{\mathbb N}} 

\title{A Hamilton-Jacobi approach of sensitivity of ODE flows and switching points in optimal control problems}

\author{{
V. Riquelme$^{1}$}\\[2mm]
$^{1}$ Departamento de Matem\'atica, Universidad T\'ecnica Federico Santa Mar\'ia,\\Avenida Espa\~na 1680, Valpara\'iso, Chile\\[2mm]
{\tt 
victor.riquelmef@usm.cl}\\[2mm]
}

\begin{document}
\maketitle

\begin{abstract}

In optimal control problems of control-affine systems, whose solutions are bang-bang or singular type, verification of optimality using the Hamilton-Jacobi-Bellman (HJB) equation involves the computation of partial derivatives of  switching times and switching states  with respect to initial conditions (time and state). 
In this paper, we establish a formula for the partial derivatives of ordinary differential equations (ODE) flows with respect to initial conditions, which is more suitable for using in HJB equation than such provided by the classical theory of ODE. We apply the obtained results to the sensitivity analysis of  hitting time and state of a reachable set, that in an optimal control problem can represent a switching locus.

\medskip


\end{abstract}

\if{
\emph{Highlights:}
\begin{itemize}
\item We derive a sensitivity formula for an ODE flow, in the form of a partial differential equation
\item We prove the differentiability, and derive sensitivity formulas, for the first hitting time of a set, under transversality conditions
\item We use the obtained results for verification of optimality, in optimal control problems, with HJB techniques 
\end{itemize}
}\fi

\medskip
\section{Introduction}

An interesting class of optimal control problems, governed by ordinary differential equations (ODE), consist in control-affine systems with control-affine cost \cite{sontag1998}, whose solutions are characterized by switching functions, obtained via Pontryagin's maximum principle \cite{clarke2013}. For this type of problems, optimal controls along the optimal trajectories consist of sequences of consecutive bangs and/or singular arcs. Along each bang arc, the control takes a constant extreme value. This procedure naturally induces feedback controls in the state-time space, which generates a patchy vector field \cite{ancona1999}. Thus,  if the optimal control is composed only by bangs, the dynamics of the problem along an optimal solution follow a sequence of uncontrolled ordinary differential equations (i.e., with constant controls), up to the first time that the solution of the system hits the boundary of each patch (which we refer to as \emph{hitting time}), that is, the corresponding switching curve. \medskip

A usual procedure to solve an optimal control problem consists in finding  extremals via Pontryagin's maximum principle \cite{clarke2013}, and then proving that the associated controls are optimal, building a verification function (candidate to value function) using these controls, and proving that  this verification function is solution of the Hamilton-Jacobi-Bellman (HJB) equation \cite{BC97,cesari1983,vinter2000} of the problem (see, for instance, \cite[Section 8.7]{bressan2007} for a discussion on the verification of the sufficient conditions for a control to be optimal). For control-affine problems, this last step involves the computation of the partial derivatives of the times and states of switch, which in this case, correspond to the hitting times and hitting states of the boundaries of each patch. Since the dynamics of the problem, up to the hitting time of the boundary of the patch, follow an uncontrolled differential equation, the problem of finding the partial derivatives of the hitting times is reduced to compute the partial derivative of the first hitting time and state of the solution of an ODE with respect to its initial data (both time and state). \medskip

Sensitivity analysis and sufficient conditions of optimality of the switching times, in optimal control problems, has been derived with basis on the classical variational equation associated to an ODE flow, recalled in  \eqref{eq:ODE_dflujo_dx} below \cite{kim,vossen}.  However, this formulation is not suitable enough for the verification of optimality using  HJB equation, since the solution of this variational equation involves the computation of the resolvent of a non-autonomous ODE, depending on the solution of the original system. In this work, we present a different approach to obtain sensitivity formulas, based on the Calculus of Variations, that are better suited to the context of HJB techniques.  \medskip

This article is structured as follows. In Section \ref{sec:sensitivity} we study the sentitivity of the solution of an ODE with respect to the initial data. In Section \ref{sec:hittingtimes}, we apply the obtained results to study the sensitivity of the first hitting time of a set, which we define as the first time that a trajectory, solution of an ODE, hits a specified set. We also derive formulas for the sensitivity of the first hitting state. In Section \ref{sec:ejemplo_gen}, we illustrate how to apply the obtained results to an optimal control problem. This method of application is then discussed in Section \ref{sec:conclusions}.

\medskip
\section{Sensitivity of an ODE flow}\label{sec:sensitivity}

Consider the autonomous ordinary differential equation
\begin{equation}\label{eq:ODE_nd}
\dot x = F(x)
\end{equation} 
given  by the vector field  $F:O\subseteq\R^n\rightarrow\R^n$ defined on an open set $O\subseteq\R^n$. If $F$ is $\mathcal C^1$, for $(x_0,t_0)\in O\times\R$ there exists an open interval $I\subseteq\R$ containing $t_0$ and a unique solution $x:I\rightarrow\R^n $ of \eqref{eq:ODE_nd} with $x(t_0)=x_0$ \cite{hirschsmale2013}. Under these assumptions, the flow $x(t;x_0,t_0)$ is defined as the value of the unique solution of \eqref{eq:ODE_nd} at time $t$, with $x(t_0)=x_0$. When $I=\R$ for all $x_0\in O=\R^n$, we will say that the flow is global. In this case, the flow is continuously differentiable with respect the initial state $x_0$, and its differential $D_{x_0}x(t;x_0,t_0)$  is solution of the following variational equation \cite{hirschsmale2013}
\begin{equation}\label{eq:ODE_dflujo_dx}
\left\{\quad\begin{split}
\frac{d}{dt}D_{x_0}x(t;x_0,t_0) =&\, DF(x(t;x_0,t_0))D_{x_0}x(t;x_0,t_0),\quad t>t_0,\\[2mm]
D_{x_0}x(t_0;x_0,t_0) =&\, I_n,
\end{split}\right.
\end{equation}
where $I_n$ denotes the $n-$dimensional identity matrix and $DF$ the Jacobian of the vector field $F$. 
\medskip

In what follows, we present a different approach to sensitivity analysis. Through this article, we assume the following hypothesis:

\begin{assumption}\label{ass:global}
Suppose that $F:\R^n\rightarrow\R^n$, of class $\mathcal C^1$, defines a global flow $x(t;x_0,t_0)=(x_i(t;x_0,t_0))_{i=1\dots n}$. 
\end{assumption}
\smallskip

Under the above assumption, we establish our main result for the sensitivity analysis of system \eqref{eq:ODE_nd}.
\begin{proposition}\label{prop:sensibilidad_nd}
Suppose Assumption \ref{ass:global} holds. Then, $x(t;x_0,t_0)$ is differentiable with respect to $(x_0,t_0)$, and for all $t\in\R$,
\begin{equation}\label{eq:ODE_PDE_nd}
\frac{\partial x}{\partial t_0}(t;x_0,t_0) +  D_{x_0}x(t;x_0,t_0)F(x_0) = 0,\qquad (x_0,t_0)\in (-\infty,t)\times\R^n.
\end{equation}
\end{proposition}

\begin{proof}
The differentiability of $x(t;x_0,t_0)$ with respect to $(x_0,t_0)$ is obtained from the Theorem of smoothness of flows with respect to initial conditions, and the fact that the dynamics is autonomous \cite{hirschsmale2013}.\medskip

Now, consider a fixed time $t$, index $i\in\{1,\dots,n\}$, and the family of (uncontrolled) problems parametrized by $(x_0,t_0)$, with $t_0\leq t$:
\begin{equation}\label{eq:problema_v}
v_{t,i}(x_0,t_0) = \min\{\, \Phi_i(x(t)) = x_i(t)\,|\, \dot x(s) = F(x(s)), \, s>t_0;\, x(t_0)=x_0\,\}.
\end{equation}

Uniqueness of solutions of \eqref{eq:ODE_nd} implies that $v_{t,i}(x_0,t_0)=x_i(t;x_0,t_0)$ (the i-th coordinate of $x(t;x_0,t_0)$). On the other hand, Problem \eqref{eq:problema_v} can be viewed in the framework of Calculus of Variations. Define the Hamiltonian $H(x,p) = \langle p, F(x)\rangle$ (where $\langle\cdot,\cdot\rangle$ denotes the standard inner product in $\R^n$), and denote $x_0=(x_{0,1},\dots,x_{0,n})$. Thus, $v_{t,i}(\cdot)$ satisfies the Hamilton-Jacobi equation \cite{vinter2000}
\begin{equation}\label{eq:PDE_ODE}
\frac{\partial v_{t,i}}{\partial t_0}(x_0,t_0) + H\left(x_0,D_{x_0}v_{t,i}(x_0,t_0)\right) = 0,\quad t_0<t;\quad v_{t,i}(x_0,t) = x_{0,i},
\end{equation}
which translates into
\begin{equation}\label{eq:PDE_ODE_F_gen}
\frac{\partial x_i}{\partial t_0}(t;x_0,t_0) +  \langle D_{x_0}x_i(t;x_0,t_0), F(x_0)\rangle = 0,\quad t_0<t;\quad x_i(t;x_0,t) = x_{0,i}.
\end{equation}
Notice that \eqref{eq:PDE_ODE_F_gen} is the componentwise version of \eqref{eq:ODE_PDE_nd}, along with the condition $x(t;x_0,t)=x_0$, which proves the proposition.
\end{proof}
\medskip

\begin{remark}

Notice that the sensitivity via the variational equation \eqref{eq:ODE_dflujo_dx} follows a Lagrangian (trajectorial) approach: we follow the evolution of a given trajectory and compute how the sensitivity of said trajectory evolves along the time. Instead, in Proposition \ref{prop:sensibilidad_nd}, we propose a Hamiltonian approach, in which the sensitivity at time $t$ is obtained for all initial conditions $(x_0,t_0)$. The trajectories obtained following the Lagrangian approach would correspond to the characteristic curves of the Hamiltonian approach.\medskip

Notice that the sensitivity formula \eqref{eq:ODE_PDE_nd} in Proposition \ref{prop:sensibilidad_nd} does not require the Jacobian matrix $DF$, but only the vector field $F$.
\end{remark}
\medskip

Two immediate corollaries can be obtained from Proposition \ref{prop:sensibilidad_nd}.
\medskip

\begin{corollary}\label{cor:sensibilidad_nd}
Under the same hypotheses than Proposition \ref{prop:sensibilidad_nd}, we have 
\begin{equation}\label{eq:dflujo_dx0_gen_nd}
D_{x_0}x(t;x_0,t_0) F(x_0) = F(x(t;x_0,t_0)).
\end{equation}
\end{corollary}

\begin{proof}
Since $x(\cdot)$ is the solution of an autonomous system, it holds
\begin{equation}\label{eq:dflujo_dt0_gen}
\frac{\partial x}{\partial t_0}(t;x_0,t_0) = -\frac{\partial x}{\partial t}(t;x_0,t_0) = -F(x(t;x_0,t_0)).
\end{equation}

Combining \eqref{eq:ODE_PDE_nd} and \eqref{eq:dflujo_dt0_gen}, we conclude the result.
\end{proof}
\medskip

For the case of one-dimensional systems, we obtain:
\begin{corollary}\label{cor:sensibilidad_1d}
Suppose $F:\R\rightarrow\R$ of class $\mathcal C^1$, such that Assumption \ref{ass:global} is satisfied. Let $x_0\in\R$ such that $F(x_0)\neq0$. Then, for all $t > t_0$,
\begin{equation}\label{eq:dflujo_dx0_1d}
\frac{\partial x}{\partial x_0}(t;x_0,t_0) = \frac{F(x(t;x_0,t_0))}{F(x_0)}.
\end{equation}
\end{corollary}

\begin{proof}
From Corollary \ref{cor:sensibilidad_nd}, we directly obtain
\begin{equation}\label{eq:dflujo_dx0}
F(x_0) \frac{\partial x}{\partial x_0}(t;x_0,t_0) = F(x(t;x_0,t_0)).
\end{equation}

Since $F(x_0)\neq0$, dividing by $F(x_0)$ we conclude the result.
\end{proof}
\medskip
\begin{remark}
In the one-dimensional case, when $F(x_0)=0$, 
Corollary \ref{cor:sensibilidad_1d} does not give much information.
However, since $x_0$ is an equilibrium of the dynamics, for all $t\geq t_0$ we have $x(t;x_0,t_0) = x_0$; using \eqref{eq:ODE_dflujo_dx} we get 
\begin{equation}
\frac{\partial x}{\partial x_0}(t;x_0,t_0) = e^{F'(x_0)(t-t_0)}.
\end{equation}

\end{remark}

\medskip
\section{Application to hitting times}\label{sec:hittingtimes}

In this section, we apply the results of the previous section to hitting times on the $(x,t)$ space. 
For this, we define the hitting time of a set by the solution of an ODE, and then we show that the hitting time (and the corresponding hitting state) follow a sensitivity relation given by a PDE, similar to that shown in Proposition \ref{prop:sensibilidad_nd}.
\medskip

\begin{definition}
Let $S\subseteq\R^{n+1}$ be a closed set, 
and consider $x(t;x_0,t_0)$ the solution of \eqref{eq:ODE_nd} with initial condition $x(t_0;x_0,t_0)=x_0\in\R^n$. We define the first hitting time of $S$ by $x(\cdot)$ starting from $(x_0,t_0)$ as
\begin{equation}\label{eq:hitting_time}
\hat t_S(x_0,t_0) := \inf\{t> t_0\,|\, (x(t;x_0,t_0),t)\in S \} \,\in\R\cup\{+\infty\}. 
\end{equation}
with the convention $\inf\emptyset=+\infty$.\medskip

We say that $S$ is reachable by $x(\cdot)$ from $(x_0,t_0)\in\R^{n+1}$ if $\hat t_S(x_0,t_0)<+\infty$.
\medskip

If S is reachable by $x(\cdot)$ from $(x_0,t_0)$, we define the hitting state as
\begin{equation}\label{eq:hitting_state}
\hat x_S(x_0,t_0) := x(\hat t_S(x_0,t_0);x_0,t_0).
\end{equation}
\end{definition}
\medskip

Suppose that the set $S\subseteq\R^{n+1}$ is defined as the zero level set of a $\mathcal C^1$ mapping $G:\R^{n+1}\rightarrow\R$, that is, $S=\left\{ (x,t)\in\R^{n+1}\,|\, G(x,t)=0 \right\}$.  
Define 
\begin{equation*}
\hat S = \left\{(x,t)\in S\,\left|\, \frac{\partial G}{\partial t}(x,t) + \langle D_xG(x,t),F(x)\rangle \neq 0 \right. \right\}.
\end{equation*}

\begin{remark}
If $x(\cdot)$ intersects $S$ at $(\hat x_S,\hat t_S)\in\hat S$, then $DG(\hat x_S,\hat t_S)\neq0$. As $DG(\hat x_S,\hat t_S)$ is the normal vector to $S$ at $(\hat x_S,\hat t_S)$, this implies that $S$ is intersected transversally at $(\hat x_S,\hat t_S)$ (c.f. \cite[Section 2.4]{bressan2007}). 

\end{remark}
\medskip


Regarding the differentiability of the hitting times $\hat t_S(\cdot)$ and hitting states $\hat x_S(\cdot)$, we have the following proposition:
\medskip

\begin{proposition}\label{prop:sensitivity_hitting}
Suppose that $S\subseteq\R^{n+1}$ is defined as a zero level set of a $\mathcal C^1$ function $G:\R^{n+1}\rightarrow \R$. Define 
\begin{equation*}
\mathcal R_S^T = \left\{ (x_0,t_0)\in\R^{n+1}\,\left|\, \substack{\displaystyle (x_0,t_0)\notin S,\, S \mbox{ is reachable by $x(\cdot)$ from }(x_0,t_0),\if{\hat t_S(x_0,t_0)<+\infty,}\fi\, \\ \displaystyle x(\cdot) \mbox{ intersects $S$ at } (\hat x_S(x_0,t_0),\hat t_S(x_0,t_0))\in\hat S} \right.\right\}
\end{equation*}

Then, $\mathcal R_S^T$ is open, and $\hat t_S(\cdot)$, $\hat x_S(\cdot)$ are differentiable on $\mathcal R_S^T$. Moreover, on $\mathcal R_S^T$, we have the relations 
\begin{eqnarray}
\frac{\partial \hat t_S}{\partial t_0}(x_0,t_0) + \langle D_{x_0} \hat t_S(x_0,t_0), F(x_0)\rangle &=& 0, \label{eq:edp_tS_1}\\[2mm]
\frac{\partial \hat x_S}{\partial t_0}(x_0,t_0) + D_{x_0} \hat x_S(x_0,t_0) F(x_0) &=& 0. \label{eq:edp_xS_1}
\end{eqnarray}
\end{proposition}
\medskip

\begin{proof}

Since $S=G^{-1}(\{0\})$, with $G(\cdot)$ continuous, then $S$ is closed, and its complement $S^c$ is open. Let $(x_0,t_0)\in \mathcal R_S^T\subseteq S^{c}$. Then $\hat t_S(x_0,t_0)<+\infty$, and 
$\hat t_S(x_0,t_0)$ is a solution (with respect to $t$) of the equation 
\begin{equation}
\Gamma(t,x_0,t_0) := G(x(t;x_0,t_0),t) = 0.
\end{equation}


As $\Gamma(\cdot)$ is the composition of continuously differentiable functions, $\Gamma(\cdot)$ is $\mathcal C^1$ in a neighborhood of $(\hat t_S(x_0,t_0),x_0,t_0)$, and 
\begin{equation}
\begin{split}
\frac{\partial \Gamma}{\partial t}(\hat t_S(x_0,t_0),x_0,t_0) =&\, D_x G(\hat x_S(x_0,t_0),\hat t_S(x_0,t_0))F(\hat x_S(x_0,t_0)) + \frac{\partial G}{\partial t}(\hat x_S(x_0,t_0),\hat t_S(x_0,t_0)),
\end{split}
\end{equation} 
which is not null, because $(\hat x_S(x_0,t_0),\hat t_S(x_0,t_0))\in \hat S$. Thus, by the Implicit Function theorem \cite{cartan1983}, there exists an open ball $B$, neighborhood  of $(x_0,t_0)$, not intersecting $S$, and an open set $A$ containing $\hat t_S(x_0,t_0)$, such that a function $\tau_S(\cdot)$ can be defined as the unique $\mathcal C^1$ function $\tau_S:B\rightarrow A$ satisfying 
\begin{equation}\label{eq:aux_tfi_0}
\Gamma(\tau_S(x,t),x,t)=0,\quad (x,t)\in B,\quad \tau_S(x_0,t_0)=\hat t_S(x_0,t_0),
\end{equation}
and
\begin{equation}\label{eq:aux_tfi}
D \tau_S(\cdot) = \left.-\frac{D_{(x_0,t_0)} \Gamma(t,\cdot)}{\frac{\partial \Gamma}{\partial t}(t,\cdot)}\right|_{t=\tau_S(\cdot)}\quad \mbox{ on } B.
\end{equation}
Moreover, for every $(t^*,x_0^*,t_0^*)\in A\times B$ such that $\Gamma(t^*,x_0^*,t_0^*)=0$, we have $t^*=\tau_S(x_0^*,t_0^*)$.\medskip

In particular, this implies that $\hat t_S(\cdot)\leq \tau_S(\cdot)<+\infty$ on $B$. We now prove that $\hat t_S(\cdot)=\tau_S(\cdot)$ on some open ball $B'\subseteq B$ containing $(x_0,t_0)$.
\medskip

Suppose, by contradiction, that $\tau_S(\cdot)\neq \hat t_S(\cdot)$ in every open ball around $(x_0,t_0)$. Then, there exist $\delta>0$ such that for every $n\in\N$ there exists $(x_0^n,t_0^n)\in B\cap B((x_0,t_0),1/n)$ with 
\begin{equation}\label{eq:auxxx_01}
\tau_S(x_0^n,t_0^n)-\hat t_S(x_0^n,t_0^n)>\delta>0.
\end{equation}
The existence of such $\delta$ is guaranteed by the Implicit Function theorem. Indeed, for a point $(x_0^n,t_0^n)$ to satisfy $\tau_S(x_0^n,t_0^n)>\hat  t_S(x_0^n,t_0^n)$, it is necessary that $\hat  t_S(x_0^n,t_0^n)\notin A$. Since $\tau_S(x_0^n,t_0^n)\in A$, with $\tau_S(\cdot)$ continuous, for $n$ large enough, we have $|\tau_S(x_0^n,t_0^n)-\tau_S(x_0,t_0)|<\mathrm{dist}(\tau_S(x_0,t_0),A^c)/2$. We conclude the existence of such $\delta<\mathrm{dist}(\tau_S(x_0,t_0),A^c)/2$.\medskip

Since $\tau_S(\cdot)$ is continuous, for $\epsilon>0$ there exists $n_{\epsilon}\in\N$ such that
\begin{equation*}
-\infty<\inf\{t\,|\,(x,t)\in B\} \leq t_0^n\leq \hat t_S(x_0^n,t_0^n)<\tau_S(x_0^n,t_0^n)< \tau_S(x_0,t_0)+\epsilon<\infty,\quad \forall n\geq n_\epsilon,
\end{equation*} 
which means that the sequence $(\hat t_S(x_0^n,t_0^n))_{n\in\N}$ is bounded and, thus, converges (up to a subsequence) to a finite value $\hat t^{\star}$. Taking limits in \eqref{eq:auxxx_01}, using that $\tau_S(x_0,t_0)=\hat t_S(x_0)$,
\begin{equation}\label{eq:auxxx_02}
\hat t_S(x_0,t_0)-\hat t^{\star}\geq\delta>0.
\end{equation}

From the continuity of the flow $x(\cdot)$, we have that $\hat x_S(x_0^n,t_0^n) = x(\hat t_S(x_0^n,t_0^n);x_0^n,t_0^n)$ converges to $\hat x^{\star} = x(\hat t^{\star};x_0,t_0)$. Since $(\hat x_S(x_0^n,t_0^n),\hat t_S(x_0^n,t_0^n))\in S$ converges to $(\hat x^{\star},\hat t^{\star})$, with $S$ closed set, then $(\hat x^{\star},\hat t^{\star})\in S$. Thus, $(\hat x^{\star},\hat t^{\star})$ is a point in $S$, that is attained by $x(\cdot)$ from $(x_0,t_0)$ at the time $t^{\star}<\hat t_S(x_0,t_0)$. This contradicts the definition of $\hat t_S(\cdot)$. Thus, $\hat t_S(\cdot)$ coincides with $\tau_S(\cdot)$ on some open ball $B'\subseteq B$. Thus, $\hat t_S(\cdot)$ satisfies the same properties than $\tau_S(\cdot)$ as in \eqref{eq:aux_tfi_0} and \eqref{eq:aux_tfi}.
\medskip

Continuity of $x(\cdot)$ with respect to the initial conditions implies $\langle D G(\hat x_S(\cdot),\hat t_S(\cdot)),(F(\hat x_S(\cdot)),1)\rangle\neq 0 $ on $B'$, which proves that all initial conditions in $B'$ belong to $\mathcal R_S^T$. This proves that $\mathcal R_S^{T}$ is open. 
\medskip

Now, evaluating \eqref{eq:aux_tfi} in $(x_0,t_0)$ (considering that $\tau(\cdot)=\hat t_S(\cdot)$ in $B'$),
\begin{equation}\label{eq:derivada_t_hat_x0_gral}
\begin{split}
\frac{\partial \hat t_S}{\partial t_0}(x_0,t_0) = &\, -\dfrac{D_{x}G(\hat x_S(x_0,t_0),\hat t_S(x_0,t_0))\frac{\partial x}{\partial t_0}(t;x_0,t_0)|_{t=\hat t_S(x_0,t_0)}}{\langle D G(\hat x_S(x_0,t_0),\hat t_S(x_0,t_0)),(F(\hat x_S(x_0,t_0)),1)\rangle},\\[3MM]
D_{x_0} \hat t_S(x_0,t_0) = &\, -\dfrac{D_{x}G(\hat x_S(x_0,t_0),\hat t_S(x_0,t_0))D_{x_0}x(t;x_0,t_0)|_{t=\hat t_S(x_0,t_0)}}{\langle D G(\hat x_S(x_0,t_0),\hat t_S(x_0,t_0)),(F(\hat x_S(x_0,t_0)),1)\rangle}.
\end{split}
\end{equation}

Then, thanks to \eqref{eq:derivada_t_hat_x0_gral} and \eqref{eq:ODE_PDE_nd} (omiting the dependency of $\hat t_S$ and $\hat x_S$ on $(x_0,t_0)$), 
\begin{equation}\label{eq:edp_tS}
\begin{split}
\frac{\partial \hat t_S}{\partial t_0}+\langle D_{x_0}\hat t_S,F(x_0)\rangle =&\, \frac{-D_x G(\hat x_S,\hat t_S)}{\langle D G(\hat x_S,\hat t_S),(F(\hat x_S),1)\rangle} \left.\left[ \frac{\partial x}{\partial t_0}(t;x_0,t_0) + D_{x_0}x(t;x_0,t_0)F(x_0) \right]\right|_{t=\hat t_S}\\
=&\,0,
\end{split}
\end{equation}
which proves \eqref{eq:edp_tS_1}. Now, from \eqref{eq:hitting_state}, since $\hat x(\cdot)$ is composition of $\hat t_S(\cdot)$ (differentiable on $B'$) and $x(\cdot)$ (differentiable), it is differentiable on $B'$. Differentiating \eqref{eq:hitting_state} with respect to $x_0$ and $t_0$,
\begin{equation}\label{eq:hitting_state_derivatives}
\begin{split}
\frac{\partial \hat x_S}{\partial t_0}(x_0,t_0) 
=& \, F(\hat x_S(x_0,t_0))\frac{\partial \hat t_S}{\partial t_0}(x_0,t_0) + \left.\frac{\partial x}{\partial t_0}(t;x_0,t_0)\right|_{t=\hat t_S(x_0,t_0)},\\[3mm]
D_{x_0}\hat x_S(x_0,t_0) 
=& \, F(\hat x_S(x_0,t_0))D_{x_0} \hat t_S(x_0,t_0) + \left.D_{x_0}x(t;x_0,t_0)\right|_{t=\hat t_S(x_0,t_0)}.
\end{split}
\end{equation}

From \eqref{eq:hitting_state_derivatives}, \eqref{eq:edp_tS} and \eqref{eq:ODE_PDE_nd}, 
\begin{equation}\label{eq:edp_xS}
\begin{split}
\frac{\partial \hat x_S}{\partial t_0}(x_0,t_0) + D_{x_0} \hat x_S(x_0,t_0) F(x_0) =&\, F(\hat x_S(x_0,t_0))\left[ \frac{\partial \hat t_S}{\partial t_0}(x_0,t_0) + \langle D_{x_0}\hat t_S,F(x_0)\rangle \right]\\
&\, + \left.\left[ \frac{\partial x}{\partial t_0}(t;x_0,t_0) + D_{x_0}x(t;x_0,t_0)F(x_0) \right]\right|_{t=\hat t_S}\\
&=0,
\end{split}
\end{equation}
which concludes \eqref{eq:edp_xS_1}.
\end{proof}

\begin{remark}
Without the condition $(\hat x_S(x_0,t_0),\hat t_S(x_0,t_0))\in\hat S$, $\hat t_S(\cdot)$ may fail to be continuous at $(x_0,t_0)$. Consider in $\R^2$ the dynamics $\dot x_1 = 1$, $\dot x_2 = 0$, and $G((x_1,x_2),t) = -x_1^2+x_1^3-x_2^2$, whose zero level set is $S=(\{(0,0)\}\cup\{(x_1,x_2)\,|\,x_2=\pm x_1\sqrt{x_1-1},\, x_1\geq1\})\times\R$. Denote initial conditions $x_0=(x_{0,1},x_{0,2})$ at initial time $t_0$. Consider the set $D=\{((x_{0,1},0),t_0)\,|\,x_{0,1}<0,\,t_0\in\R\}$. For every $((x_{0,1},0),t_0)\in D$, we have $\hat x_S((x_{0,1},0),t_0)=(0,0)$, with $DG((0,0),t)=((0,0),0)$. If $x_{0,1}<0$, $\hat t_S((x_{0,1},0),t_0) = t_0-x_{0,1}$ and $\hat t_S((x_{0,1},x_{0,2}),t_0) \geq t_0-x_{0,1}+1$ when $x_{0,2}\neq0$. We conclude that $\hat t_S(\cdot)$ is discontinuous at every point $((x_{0,1},0),t_0)\in D$.
\end{remark}

\if{

\medskip
\section{Example: application to optimal control}\label{sec:example}

Consider the well known rocket railroad car problem \cite{BC97}
\begin{equation}\label{eq:valor_carrocohete}
V(z_0,v_0;t_0) = \min\left\{ t_f\geq t_0 \left.\,\right\vert\, \substack{\displaystyle\dot z=v,\, \dot v = u,\, u:\R\rightarrow[-1,1] \mbox{ Lebesgue measurable},\\ \displaystyle z(t_f)=v(t_f)=0,\, z(t_0)=z_0,\,v(t_0)=v_0}\right\}.
\end{equation}
Consider in this example the state variable $x=(z,v)$. Pontryagin's Maximum Principle's necessary conditions state that there exists at most one switch, with no singular arcs, for any optimal trajectory. It also gives as a candidate to optimal synthesis
\begin{equation}\label{eq:control_carrocohete}
u[z,v;t] = \left\{\begin{split}
\quad -1, &\quad [G(z,v;t)>0] \mbox{ or } [G(z,v;t)=0,\,v>0],\\
\quad 1, &\quad [G(z,v;t)<0] \mbox{ or } [G(z,v;t)=0,\,v<0],\\
\quad 0, &\quad G(z,v;t)=0,\,z=v=0,
\end{split}\right.
\end{equation} 
with $G(z,v;t) = z+v|v|/2$. Define $S:=\left\{ (z,v,t)\in\R^2\,|\,G(z,v;t)=0 \right\}$. Define also $\hat t_{S}(z_0,v_0;t_0)$ and $\hat t_{\{0\}}(z_0,v_0;t_0)$ as the hitting times of $S$ and $\{(0,0)\}\times\R$ (respectively) by the solution $x(t;z_0,v_0,t_0)$ of the dynamics driven by the feedback control $u[\cdot]$, with corresponding hitting states $\hat x_{S}(z_0,v_0;t_0)=(\hat z_S(z_0,v_0;t_0),\hat v_S(z_0,v_0;t_0))$ and $\hat x_{\{0\}}(z_0,v_0;t_0)=(0,0)$ (respectively). The terminal time of the process driven by $u[\cdot]$, starting from the point $(z_0,v_0)$ at initial time $t_0$, is
\begin{equation}\label{eq:tiempo_carrocohete}
w(z_0,v_0;t_0) = \hat t_{\{0\}}(\hat z_S(z_0,v_0;t_0),\hat v_S(z_0,v_0;t_0);\hat t_{S}(z_0,v_0;t_0))
\end{equation}

We claim the control $u[\cdot]$ defined in \eqref{eq:control_carrocohete} is optimal. Indeed, thanks to \cite[Theorem 7.3.3]{bressan2007}, it is enough to show that $w(\cdot)$, as in \eqref{eq:tiempo_carrocohete}, satisfies the HJB equation
\begin{equation}\label{eq:HJB_carrocohete}
-1 -\frac{\partial w}{\partial z_0}v_0 + \max_{u\in[-1,1]}\left\{- \frac{\partial w}{\partial v_0}u \right\} = 0,
\end{equation}
on $(\R^2\times\R)\backslash S$, with the maximizer $u$ in \eqref{eq:HJB_carrocohete} at $(z_0,v_0,t_0)$ coinciding with $u[z_0,v_0,t_0]$. 
\medskip

It is not difficult to check that every trajectory $x(\cdot;z_0,v_0,t_0)$ driven by $u[\cdot]$, with $(z_0,v_0,t_0)\notin S$, intersects $S$ transversally at $\hat x_S(z_0,v_0;t_0)$. A simple calculation yields, for $(z_0,v_0,t_0)\in S$, 
\begin{equation}\label{eq:hat_t0_S}
\hat t_{\{0\}}(z_0,v_0;t_0) = t_0 + |v_0|.
\end{equation}
Now, for $(z_0,v_0,t_0)\notin S$, a simple argument allows to say that
\begin{equation}\label{eq:signo_hatv_S}
{\rm sign}(\hat v_S(z_0,v_0;t_0)) = -{\rm sign}(G(z_0,v_0;t_0)).
\end{equation}

From \eqref{eq:tiempo_carrocohete}, combining \eqref{eq:hat_t0_S} and \eqref{eq:signo_hatv_S}, we get 
\begin{equation}\label{eq:dw}
\frac{\partial w}{\partial t_0} = \frac{\partial \hat t_S}{\partial t_0}  -{\rm sign}(G) \frac{\partial \hat v_S}{\partial t_0} ,\quad
\frac{\partial w}{\partial z_0} = \frac{\partial \hat t_S}{\partial z_0}  -{\rm sign}(G) \frac{\partial \hat v_S}{\partial z_0} ,\quad
\frac{\partial w}{\partial v_0} = \frac{\partial \hat t_S}{\partial v_0}  -{\rm sign}(G) \frac{\partial \hat v_S}{\partial v_0} .
\end{equation}

On the other hand, from Proposition \ref{prop:sensitivity_hitting}, we obtain that, if $(z_0,v_0,t_0)\notin S$, 
\begin{equation}\label{eq:auxx1}
\frac{\partial \hat t_S}{\partial z_0}v_0  = -\frac{\partial \hat t_S}{\partial t_0}  +  {\rm sign}(G)\frac{\partial \hat t_S}{\partial v_0},\quad  
\frac{\partial \hat z_S}{\partial z_0}v_0  = -\frac{\partial \hat z_S}{\partial t_0}  +  {\rm sign}(G)\frac{\partial \hat z_S}{\partial v_0},
\quad \frac{\partial \hat v_S}{\partial z_0}v_0  = -\frac{\partial \hat v_S}{\partial t_0}  +  {\rm sign}(G)\frac{\partial \hat v_S}{\partial v_0}.
\end{equation}

Replacing \eqref{eq:dw} in the left-hand side of \eqref{eq:HJB_carrocohete}, and using \eqref{eq:auxx1}, \eqref{eq:HJB_carrocohete} is equivalent to 
\begin{equation}\label{eq:HJB_carrocohete_eq}
-1+\frac{\partial \hat t_S}{\partial t_0}-{\rm sign}(G)\frac{\partial \hat v_S}{\partial t_0} - {\rm sign}(G)\frac{\partial \hat t_S}{\partial v_0} + \frac{\partial \hat v_S}{\partial v_0}  + \max_{u\in[-1,1]}\left\{ -u\left(  \frac{\partial \hat t_S}{\partial v_0} - {\rm sign}(G)\frac{\partial \hat v_S}{\partial v_0} \right) \right \} = 0.
\end{equation}

Thus, the HJB equation on $w$ is translated into a PDE on the hitting time $\hat t_S(\cdot)$ and state $\hat x_S(\cdot)=(\hat z_S(\cdot),\hat v_S(\cdot))$, without further knowledge of their explicit expression. To prove the optimality of $u[\cdot]$ (that is, that \eqref{eq:HJB_carrocohete_eq} is satisfied), it is enough to show that it coincides with the maximizer of the expression between braces in \eqref{eq:HJB_carrocohete_eq}, and that $w(\cdot)$ satisfies \eqref{eq:HJB_carrocohete_eq}. Summarizing, these two conditions are
\begin{eqnarray}
{\rm sign}\left(\frac{\partial \hat t_S}{\partial v_0} - {\rm sign}(G)\frac{\partial \hat v_S}{\partial v_0}\right) &=& {\rm sign}(G) \label{eq:cond_max}\\
-1+\frac{\partial \hat t_S}{\partial t_0}-{\rm sign}(G)\frac{\partial \hat v_S}{\partial t_0}  &=& 0 \label{eq:cond_hjb}
\end{eqnarray}
\medskip

A direct calculation gives, for $(z_0,v_0,t_0)\notin S$, 
\begin{equation}\label{eq:auxx3}
\hat t_{S}(z_0,v_0;t_0) = \left\{\begin{split}
\quad t_0+v_0+\sqrt{\frac{v_0^2}{2}+z_0},\quad G(z_0,v_0;t_0)>0,\\
\quad t_0-v_0+\sqrt{\frac{v_0^2}{2}-z_0},\quad G(z_0,v_0;t_0)<0,
\end{split}\right.
\end{equation}
and 
\begin{equation}\label{eq:auxx4}
\left(\hat z_{S}(z_0,v_0;t_0),\hat v_{S}(z_0,v_0;t_0)\right) = \left\{\begin{split}
\quad \left( \frac{1}{2}\left( z_0+\frac{v_0^2}{2} \right) , -\sqrt{ z_0+\frac{v_0^2}{2} } \right),\quad G(z_0,v_0;t_0)>0,\\
\quad \left( -\frac{1}{2}\left( \frac{v_0^2}{2}-z_0 \right) , \sqrt{ \frac{v_0^2}{2}-z_0 } \right),\quad G(z_0,v_0;t_0)<0,
\end{split}\right.
\end{equation}

Hence, from \eqref{eq:auxx3} and \eqref{eq:auxx4}, \eqref{eq:cond_hjb} is trivially checked, and it is not difficult to check that the sign of $G(\cdot)$ determines the sign of the left-hand side of \eqref{eq:cond_max}. Then, \eqref{eq:cond_max} is also satisfied. As conclusion, $w(\cdot)$ coincides with $V(\cdot)$, and all the trajectories driven by $u[\cdot]$ are optimal.

}\fi

\medskip
\section{Example of application to optimal control}\label{sec:ejemplo_gen}

In this section, we illustrate how to use the obtained results in an optimal control problem. Consider the one-dimensional system 
\begin{equation}\label{eq:sistema_ejemplo}
\dot x(t) = f(x(t)) + u(t) g(x(t)), \quad t>t_0,\qquad x(t_0)=x_0,
\end{equation}
with $u(\cdot)$ in the set of admissible controls
\begin{equation}\label{eq:control_admisible_ejemplo}
\mathcal U := \left\{ u:[0,T]\rightarrow [-1,1] \,|\, u(\cdot) \mbox{ Lebesgue measurable} \right\}.
\end{equation} 
Suppose we are interested in the minimization, with respect to $u(\cdot)\in\mathcal U$, of the fixed final time cost functional
\begin{equation}\label{eq:costo_ejemplo}
J(x_0,t_0;u(\cdot)) = \int_{t_0}^T l_x(x(t)) + l_u(x(t))u(t)\,dt
\end{equation}
with $x(\cdot)$ solution of \eqref{eq:sistema_ejemplo} associated to $u(\cdot)$. Define the value function associated to the family of problems with different initial conditions: 
\begin{equation}\label{eq:valor_ejemplo}
V(x_0,t_0) = \min\left\{ J(x_0,t_0;u(\cdot)) \,|\, x(\cdot) \mbox{ solution of \eqref{eq:sistema_ejemplo}},\, u(\cdot)\in \mathcal U\right\}.
\end{equation}
Under standard assumptions on $f(\cdot),g(\cdot),l_x(\cdot),l_u(\cdot)$ for existence of solutions of the optimal control problems that define $V(\cdot)$, $V(\cdot)$ is a viscosity solution of the HJB equation
\begin{equation}\label{eq:hjb_ejemplo}
-\frac{\partial V}{\partial t_0} -l_x(x_0) - f(x_0)\frac{\partial V}{\partial x_0} + \max_{u\in [-1,1]} \left\{  -u\left(l_u(x_0) + g(x_0)\frac{\partial V}{\partial x_0}\right) \right\} = 0, \quad V(\cdot,T) = 0.
\end{equation}

The system is control-affine. Suppose that, via Pontryagin's maximum principle, we are able to prove that the optimal controls are of bang-bang type, with at most one switch, and $u=-1$ during the last part of the trajectory. Suppose that a switch locus can be identified in the state-time space, as the zero level set of a $\mathcal C^1$ function $G(\cdot)$, that is, $S=\{(x,t)\in\R\times[0,T]\,|\, G(x,t)=0\}$. Define $G^{+} = \{(x,t)\in\R\times[0,T]\,|\, G(x,t)>0\}$ and $G^{-} = \{(x,t)\in\R\times[0,T]\,|\, G(x,t)<0\}$. Suppose, without loss of generality, that $\R\times\{T\}\subseteq G^{-}$, and that $G^{-}\cup S$ is positively invariant under the dynamics \eqref{eq:sistema_ejemplo} with $u=-1$. Moreover, suppose that for every $(x_0,t_0)\in G^{+}$, $x(\cdot;x_0,t_0)$ generated with constant control $u=1$ intersects $S$ transversally. 
%
%
We conjecture that the optimal feedback control is 
\begin{equation}\label{eq:control_ejemplo}
u[x,t] = \left\{\quad \begin{split}
-1,&\quad \mbox{ if } (x,t)\in G^{-}\cup S,\\
1,&\quad \mbox{ if } (x,t)\in G^{+}.
\end{split}
\right.
\end{equation}

To prove the conjecture, define $x^{+}(t;x_0,t_0)$ and $x^{-}(t;x_0,t_0)$ the solutions of \eqref{eq:sistema_ejemplo} with constant controls $u=1$ and $u=-1$, respectively. For initial time $t_d$, initial state $x_d$, and final time $t_f$, the costs 
\begin{equation}
\begin{split}
J^{-}(x_d,t_d,t_f) =&\, \int_{t_d}^{t_f} \left(\,l_x(x^{-}(t;x_d,t_d)) - l_u(x^{-}(t;x_d,t_d) )\,\right) \,dt,\\
J^{+}(x_d,t_d,t_f) =&\, \int_{t_d}^{t_f} \left(\,l_x(x^{+}(t;x_d,t_d)) + l_u(x^{+}(t;x_d,t_d) )\,\right) \,dt,
\end{split}
\end{equation}
and the candidate to value function associated to $u[\cdot]$
\begin{equation}\label{eq:cand_valor_ejemplo}
w(x_0,t_0) = \left\{\begin{split}
J^{-}(x_0,t_0,T),&\quad (x_0,t_0)\in G^{-}\cup S,\\
\quad J^{+}(x_0,t_0,\hat t_S(x_0,t_0)) + J^{-}(\hat x_S(x_0,t_0),\hat t_S(x_0,t_0),T),&\quad (x_0,t_0)\in G^{+},
\end{split}\right.
\end{equation}
where $\hat t_S(x_0,t_0)$ and $\hat x_S(x_0,t_0)$ are the hitting time and hitting state of $S$ by $x^{+}(\cdot;x_0,t_0)$. 
Thus, according to \cite[Corollary 7.3.4]{bressan2007}, it suffices to prove that $w(\cdot)$ is solution of \eqref{eq:hjb_ejemplo} in $G^{-}\cup G^{+}$.

Notice that, for every fixed $t_f$, since $J^{-}(\cdot,\cdot,t_f)$ and $J^{+}(\cdot,\cdot,t_f)$ use a constant control value, they satisfy the Hamilton-Jacobi equations
\begin{equation}\label{eq:hj_ejemplo}
\left\{\quad\begin{split}
-(l_x(x_d)-l_u(x_d))-\frac{\partial J^{-}}{\partial t_d}(x_d,t_d,t_f)-(f(x_d)-g(x_d))\frac{\partial J^{-}}{\partial x_d}(x_d,t_d,t_f) \,&= 0,\quad t_d<t_f, \\
-(l_x(x_d)+l_u(x_d))-\frac{\partial J^{+}}{\partial t_d}(x_d,t_d,t_f)-(f(x_d)+g(x_d))\frac{\partial J^{+}}{\partial x_d}(x_d,t_d,t_f) \,&= 0,\quad t_d<t_f,
\end{split}\right.
\end{equation}

with boundary conditions $J^{-}(x_d,t_f,t_f)=0$, $J^{+}(x_d,t_f,t_f)=0$, $x_d\in\R$. We also have, due to Proposition \ref{prop:sensitivity_hitting}, that for any $(x_0,t_0)\in G^{+}$, under the control $u[\cdot]$,
\begin{equation}\label{eq:sensibilidad_ejemplo}
\frac{\partial \hat t_S}{\partial t_0} + (f(x_0)+g(x_0))\frac{\partial \hat t_S}{\partial x_0} = 0,\quad \frac{\partial \hat x_S}{\partial t_0} + (f(x_0)+g(x_0))\frac{\partial \hat x_S}{\partial x_0} = 0.
\end{equation}

Denote $\xi_{0T}=(x_0,t_0,T)$, $\xi_{0S}=(x_0,t_0,\hat t_S(x_0,t_0))$, $\xi_{ST}=(\hat x_S(x_0,t_0),\hat t_S(x_0,t_0),T)$. Thus, to prove that $w(\cdot)$ is solution of \eqref{eq:hjb_ejemplo} in $G^{-}\cup G^{+}$, it is enough to prove the inequalities
\begin{equation}\label{eq:hjb_desig_ejemplo}
\begin{split}
l_u(x_0)+g(x_0)\frac{\partial J^{-}}{\partial x_i}{(\xi_{0T})} \,&> 0,\quad  (x_0,t_0)\in G^{-},\\ 
l_u(x_0) + g(x_0)\left( \frac{\partial J^{+}}{\partial x_d}(\xi_{0S}) + \frac{\partial J^{+}}{\partial t_f}(\xi_{0S})\frac{\partial \hat t_S}{\partial x_0} + \frac{\partial J^{-}}{\partial x_d}(\xi_{ST})\frac{\partial \hat x_S}{\partial x_0} + \frac{\partial J^{-}}{\partial t_d}(\xi_{ST})\frac{\partial \hat t_S}{\partial x_0} \right) \,&< 0,\quad  (x_0,t_0)\in G^{+},
\end{split}
\end{equation}
since, if these inequalities are satisfied, then \eqref{eq:hjb_desig_ejemplo}, along with \eqref{eq:hj_ejemplo} and \eqref{eq:sensibilidad_ejemplo}, give the desired result.\medskip

It is worth to remark that, at this stage, it is necessary to compute derivatives of the hitting times and states with respect to fewer variables. In the previous expression, we have chosen to keep the derivatives of the hitting times and states with respect to $x_0$, but, from \eqref{eq:sensibilidad_ejemplo}, we could have chosen to express  \eqref{eq:hjb_desig_ejemplo} in terms of the derivatives with respect to $t_0$ instead, depending on which are easier to compute.  Also, depending of the functions involved, some of the derivatives of $J^{+}$ and $J^{-}$ can be more easily obtained; for instance, 
\begin{equation*}
\frac{\partial J^{+}}{\partial t_f}(\xi_{0S}) = l_x(\hat x_S(x_0,t_0)) + l_u(\hat x_S(x_0,t_0)).
\end{equation*}


\medskip
\section{Discussion}\label{sec:conclusions}

In this article, we derive a sensitivity formula for the solutions of an ODE with respect to initial data. Using the obtained result, we prove the differentiability of the hitting times of a set (given as the zero level set of a $\mathcal C^1$ function) by the solution of an ODE, under transversality conditions and non vanishing gradient at the hitting point. This differentiability result is not true in general, since the hitting times are defined as the minimum time function of attaining a set. In general, Lipschitz-continuity results for the minimum time function can be found, in the context of optimal control, using the Petrov's inward pointing condition with respect to the target set \cite{BC97,cannarsa2004}. Nevertheless, in uncontrolled systems, this condition cannot be assured. 
We also obtain a sensitivity relation for the first hitting time and the corresponding hitting state.
\smallskip

These results can be applied to the verification of optimality of a control, in control-affine problems, with fixed final time or free final time. 
%
Suppose that Pontryagin's principle allows us to characterize the behavior of extremals (among which the optimal solutions can be found, should they exist), identifying different possible switching locus and feedback controls in the $(x,t)$ space. Thus, for each of these (feedback) controls associated to a Pontryagin extremal, a candidate of value function (optimal cost depending on initial data) can be obtained via integration of the running cost, depending on the trajectory of the ODE under the corresponding control. This cost will depend of the state and time of the switches, corresponding to the hitting times and states of some set (typically, the switching locus). Thus, to verify that the candidate to value function is, in fact, optimal, we can verify that it is solution of the HJB equation in each patch (e.g. \cite[Theorem 7.3.3, Corollary 7.3.4]{bressan2007}). Thus, it will be necessary to compute the derivatives of the hitting times and states with respect to their respective initial conditions. The results of the present work suit well this procedure, by simplifying these computations (even without knowing the explicit formula of these hitting times and states). This procedure is illustrated in Section \ref{sec:ejemplo_gen}.


\smallskip


\section{Acknowledgments}
This work was supported by ANID, Chile, through project ANID FONDECYT 3180367. The author thanks Prof. Pedro Gajardo for fruitful discussions and useful comments, and Prof. Alain Rapaport for important remarks and clarifying examples.

\section{Conflict of interests}
The author declares that there is no conflict of interest in relation to the
results of this paper.


\bibliographystyle{plain}
\bibliography{biblio}



\end{document}